\documentclass{article}
\usepackage[all]{xy}
\usepackage{amssymb,amsthm}
\usepackage{amsmath}
\usepackage[latin5]{inputenc}
\usepackage{graphicx}
\usepackage{subcaption}
\usepackage{caption}
\newtheorem{example}{Example}

\newtheorem{prop}{Proposition}
\title{Fractals as Sculptures}
\author{Şahin Koçak \footnote{Anadolu University. e-mail: skocak@anadolu.edu.tr} }

\begin{document}
\maketitle

\begin{abstract}

The attractors of iterated function systems are usually obtained as the Hausdorff limit of any non-empty compact subset under iteration. In this note we show that an iterated function system on a boundedly compact metric space has compact, invariant subspaces so that the attractor of the IFS can also be expressed as the intersection of a sequence of decreasing compact spaces.

\end{abstract}

\textbf{Keywords:} Iterated function systems

\textbf{MSC:} 28A80

Painting and poetry have more to do with ``adding", but sculpture has more to do with ``subtracting". As the legend goes, Michelangelo must have said to the astounded spectators of David, ``He was already in the marble block, I just had to chisel away the superfluous material".

\begin{figure}[h]
\centering
\includegraphics[scale=0.4]{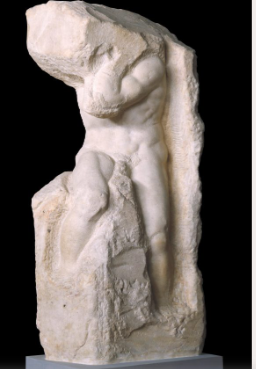}
\caption{\label{mich}Michelangelo chiseling away the superfluous material (\cite{Michweb})}
\end{figure}

I remembered Michelangelo's words as I for the first time encountered the Cantor set in my youth. Cantor was taking his thin and fragile block  and chiseling away some superfluous material. He was first
discarding the (open) middle third, without hesitating to make his block disconnected (in contrast to sculptors whose works are mostly connected), then discarding the middle thirds of the remaining two pieces, and was going on forever applying this recipe. What remained in his hands, was utterly disconnected.
We called it later a fractal! These points surviving all the deletions are the intersection of all the intermediary stages.
(see Fig.\ref{cantor})
\begin{figure}[h]
\centering
  \includegraphics[width=0.9\textwidth]{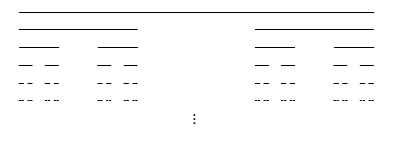}
\caption{Cantor Set obtained by successive deletions}
  \label{cantor}
\end{figure}

Then Sierpinski took a triangle, discarded the (open) middle fourth, then the middle fourths of the remaining three, and so on... The surviving points gave another fractal.
 (see Fig.\ref{sier})
\vspace{-.22cm}
\begin{figure}[h]
\centering
  \includegraphics[width=\textwidth]{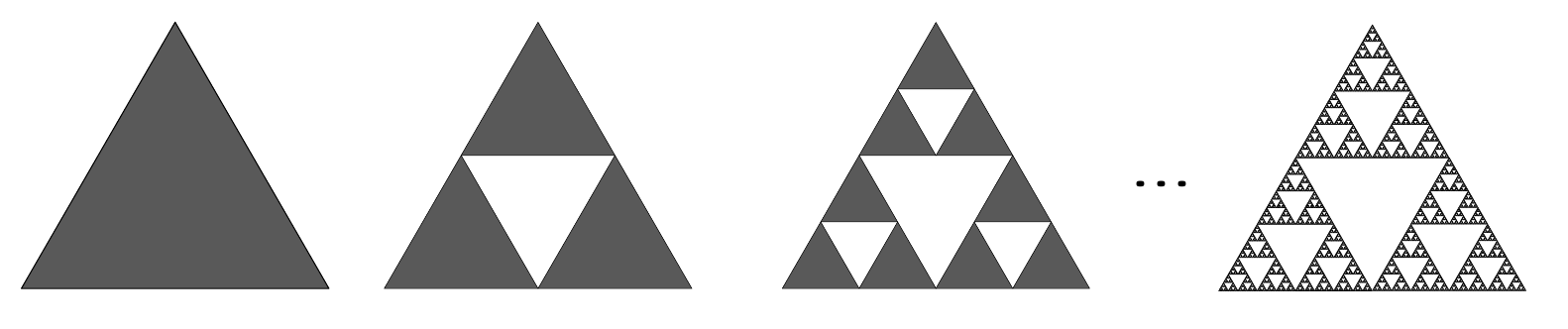}
 \vspace{-.8cm} \caption{Sierpinski Triangle obtained by successive deletions }
  \label{sier}
\end{figure}

Menger came maybe closer to the great masters of art as he took a 3-dimensional cube and chiseled away tunnels through it. (see Fig.\ref{Men}) This marvellous fractal having zero volume contains nevertheless a homeomorphic copy of every compact topological space of topological dimension 1. It has been very recently (\cite{Knots}) shown that it even contains a copy of every knot!

\begin{figure}[h]
\centering
\includegraphics[scale=0.7]{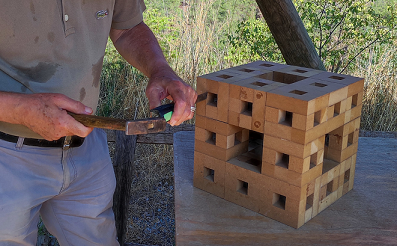}
\caption{\label{Men}Chiseling tunnels through a cube}
\end{figure}

In the seventies there were a large collection of ``fractals" around and the term was coined and the notion was made popular to the public by Mandelbrot (\cite{Man1}, \cite{Man2}). Obviously a systematic treatment of these creatures was in need and Hutchinson devised the notion of iterated function systems (\cite{Hutch}), which comprised many of these fractals; though a subspecies emerging from a totally different source, from some innocent-seeming dynamical systems, resisted to pass into this scheme. Nevertheless, the iterated function systems were able to mimic many ``fractals", including even botanical shapes (\cite{Barn}) and we became addicted to the newly upcoming computers
to see the slowly emerging, startling fractals. You consider some
contractions $f_1,f_2,\dots,f_N$ on a complete metric space $X$
and show that there exists a unique (non-empty) compact subspace
$A$ satisfying $f_1(A)\cup f_2(A) \cup \dots \cup f_N(A)=A$, which
you declare to be a ``fractal" (we use this term in a loose sense,
calling to mind spaces with some kind of self-similarity and
fractional dimension). The iterated function system (IFS) approach
is constructive in the sense that you can start with any
(non-empty) compact subspace $B\subset X$, build the image
$F(B)=f_1(B)\cup f_2(B)\cup \dots \cup f_N(B)$, then iterate this
process obtaining a sequence
                 \[B, F(B), F(F(B))=:F^2(B),\dots, F^n(B),\dots\]
of which the longed-for fractal $A$ is the limit point with respect to the Hausdorff metric. This metric is meanwhile an indispensible tool for pattern recognizers! (For a lovely exposition of the details we refer to \cite{Barn}.)

The only unhappy side of this construction is that it does not
need to hold
\[B\supset F(B) \supset F^2(B) \supset \cdots F^n(B) \supset \cdots \ .\]

If this would be the case, the infinite intersection would be the
Hausdorff limit and we could imagine like a sculptor that this
limit is obtained by successive deletions from the initial compact
set $B$. But generally this is not true. However we will show
below that in fact such a compact subspace $B$, which is invariant under all the functions constituting the iterated function system  can always be
found, thus somehow reconciling the IFS approach with the
classical deletion approach.

Our candidate of such a compact subspace $B$ will be a so-called multi-foci ellipse, a notion which might have interested the mathematicians of the antiquity, but apparently was  investigated only in the 19. century by a teen no lesser than Maxwell, the creator of the famous electro-magnetic equations (\cite{Max}). His ``oval curves" are defined as the locus of points on the plane whose distance-sum to a given (finite) set of points is a given number.  The same definition applies surely to any metric space, where we allow the distance-sum to be less than or equal to a given number. We will use in the following multi-foci ellipses whose focal points will be the fixed points of the contractive functions of the given iterated function system.

\begin{figure}[h]
\centering
\includegraphics[scale=0.5]{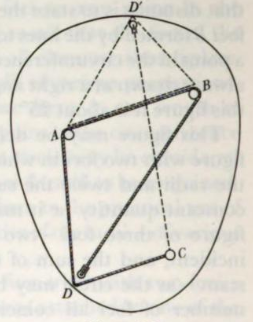}
\caption{\label{Max} Maxwell's drawing of a 3-foci ellipse (with focal points $A, B$ and $C$) with a thread (\cite{Max}) (and without a computer, which was non-existent anyway!)}
\end{figure}

 In the following proposition we will make the slightly stronger
assumption that the space $X$ is boundedly compact which suffices
for all applications and which implies completeness.

\begin{prop}
Let $(X,d)$ be a boundedly compact metric space and $f_i:X\to
X$, $i=1,2,\dots,N$ be contractions.
 Then, there exists a (non-empty) compact subspace $B\subset X$ with $f_i(B)\subset B$ for all $i=1,2,\dots,N.$

\end{prop}

\begin{proof}
Let the contractions $f_i$ have the contraction ratios $\lambda_i$
with $0<\lambda_i <1$, and $a_i\in X \ (i=1,2,\dots, N)$ be the
fixed points of $f_i$ (i.e. $f_i(a_i)=a_i$). \linebreak Let
$\displaystyle \lambda:=\max_{1\leq i \leq N} \{\lambda_i\}, \
\displaystyle D_i:=\sum_{j=1}^N d(a_i,a_j),\  D:=\max_{1\leq i\leq
N} \{ D_i\}$ and $M:=\dfrac{1+\lambda}{1-\lambda}D$. We consider the $N-$foci ellipse
$\displaystyle B=\{x\in X\mid \sum_{j=1}^{N}d(x,a_j)\leq M\}$. For
$x\in B$, we show that $f_i(x)\in B,$ i.e. $\displaystyle
\sum_{j=1}^{N}d(f_i(x),a_j)\leq M$ for all $i=1,2,\dots, N$.

$\begin{array}{ll}
\text{For } j=i: &  d(f_i(x),a_i)=d(f_i(x),f_i(a_i))\leq \lambda d(x,a_i)\\
\text{For } j\neq i: &  d(f_i(x),a_j)\leq d(f_i(x),a_i)+d(a_i,a_j)\\
                     &\leq \lambda d(x,a_i)+d(a_i,a_j)\leq \lambda d(x,a_j)+(\lambda+1)d(a_i,a_j)
\end{array}$

\noindent Thus,

$\begin{array}{ll}
  \displaystyle \sum_{j=1}^{N}d(f_i(x),a_j)& \leq \lambda d(x,a_i)+\displaystyle \sum_{j\neq i}[\lambda d(x,a_j)+(\lambda +1)d(a_i,a_j)]\\
&\leq \lambda \left(\displaystyle \sum_{j=1}^{N}d(x,a_j)
\right)+(\lambda +1)D_i\leq \lambda M +(\lambda +1)D=M.
\end{array}$
\end{proof}

\begin{example}
We consider the  IFS on   $\mathbb{R}^2$ consisting of two contractions $f_1$ and $f_2$,
whereby $f_1$ is a rotation of angle $\pi/6$ around the origin in the counterclockwise direction
 with contraction ratio $1/2$ and $f_2$ is a rotation of angle $\pi/6$ around $(1,0)$
 in the clockwise direction with contraction ratio $3/5$. If we choose the disk with center
  $(0,0)$ and radius $1$ as a random starting
  set $B$ of iteration, for example,
  then $F(B)=f_1(B)\cup f_2(B)$ is not contained in $B$ (See
  Fig.6).
 If however we choose the starting set as the ellipse
  $\frac{(x-1/2)^2}{4}+\frac{y^2}{15/4}\leq 1$ as suggested by the proof above,
  it holds $F(B)=f_1(B)\cup f_2(B)\subset B$. (See Fig.7a. We also show the second iteration
  of this IFS in Fig.7b) It seems that our sculptor will have a difficult job of laborious deletions to get a faint impression of the emerging fractal.
  To make him/her the life easier and to give a glimpse into the fractal we can offer to use an iteration
   \'a la Hutchinson and Barnsley, starting with a singleton, say $\{(1,1)\}$ (See Fig.8).

\begin{figure}[h]
\centering
  \includegraphics[width=0.6\textwidth]{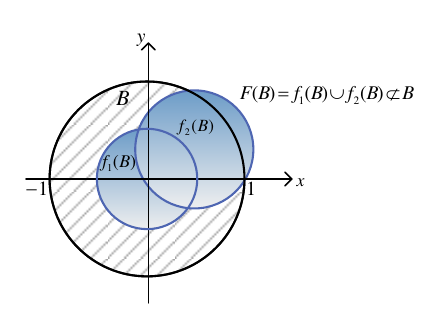}
 \vspace{-.3cm} \caption{A starting set $B$ which is not mapped into itself under $F$}
\end{figure}

\begin{figure*}[h]
    \begin{subfigure}[b]{0.45\textwidth}
        \centering
        \includegraphics[scale=.65]{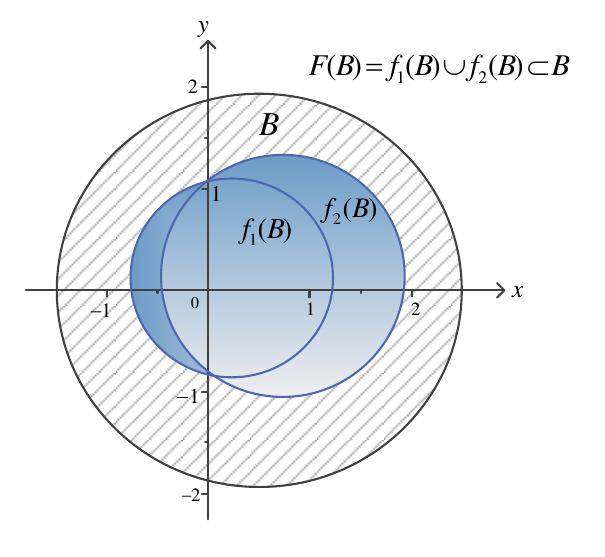}
        \caption{}
    \end{subfigure}%
    \hspace{2cm}
    \begin{subfigure}[b]{0.45\textwidth}
        \centering
        \includegraphics[scale=.65]{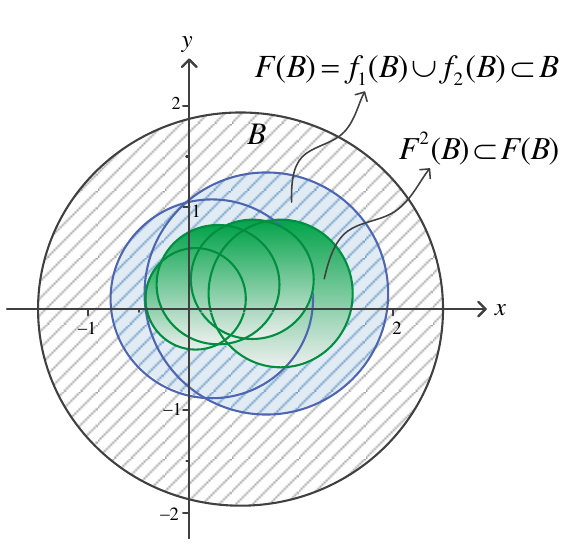}
        \caption{}
    \end{subfigure}
    \vspace{-.8cm}  \caption{A starting set $B$ which is mapped into itself under
    $F$(Fig.7a) and the second iteration $F^2(B)$ of it (Fig.7b)}
\end{figure*}


\begin{figure}
\centering
  \includegraphics[width=0.2\textwidth]{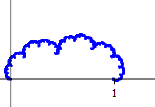}
 \vspace{-.3cm} \caption{The tenth iteration starting with a
 singleton}
\end{figure}
\end{example}

We hope that playing with the multi-foci ellipses of Maxwell for the well-known iterated function systems for the classical fractals might be intriguing for some readers versed in computer drawing. What is e.g. the locus of points having a distance-sum 2 to the corners of an equilateral triangle with unit edge-length?

\section*{Acknowledgment} I am indebted to Derya Çelik for drawing
the figures.


\end{document}